\documentclass[10pt,reqno]{amsart}

\usepackage[a4paper,left=35mm,right=35mm,top=30mm,bottom=30mm,marginpar=25mm]{geometry}

\usepackage{amsmath}
\usepackage{amssymb}
\usepackage{amsthm}
\usepackage[latin1]{inputenc}
\usepackage{eurosym}
\usepackage[dvips]{graphics}

\usepackage{graphicx}

\usepackage{epsfig}
\usepackage{hyperref}
\usepackage{dsfont}

\usepackage[displaymath,mathlines]{lineno}

\allowdisplaybreaks

\usepackage{hyperref}

\usepackage{ifthen}

\makeindex

\renewcommand{\div}{\operatorname{div}}

\newcommand{\Rr}{{\mathbb{R}}}

\newcommand{\Nn}{{\mathbb{N}}}

\newcommand{\Aa}{{\mathcal{A}}}

\newcommand{\epsi}{\epsilon}

\def\dx{{\rm d}x}

\def\leq{\leqslant}
\def\geq{\geqslant}

\numberwithin{equation}{section}

\newtheoremstyle{thmlemcorr}{10pt}{10pt}{\itshape}{}{\bfseries}{.}{10pt}{{\thmname{#1}\thmnumber{
#2}\thmnote{ (#3)}}}
\newtheoremstyle{thmlemcorr*}{10pt}{10pt}{\itshape}{}{\bfseries}{.}\newline{{\thmname{#1}\thmnumber{
#2}\thmnote{ (#3)}}}
\newtheoremstyle{defi}{10pt}{10pt}{\itshape}{}{\bfseries}{.}{10pt}{{\thmname{#1}\thmnumber{
#2}\thmnote{ (#3)}}}
\newtheoremstyle{remexample}{10pt}{10pt}{}{}{\bfseries}{.}{10pt}{{\thmname{#1}\thmnumber{
#2}\thmnote{ (#3)}}}
\newtheoremstyle{ass}{10pt}{10pt}{}{}{\bfseries}{.}{10pt}{{\thmname{#1}\thmnumber{
A#2}\thmnote{ (#3)}}}

\theoremstyle{thmlemcorr}

\newtheorem{theorem}{Theorem}
\numberwithin{theorem}{section}

\theoremstyle{thmlemcorr*}
\newtheorem{theorem*}{Theorem}
\newtheorem{lemma*}[theorem]{Lemma}
\newtheorem{corollary*}[theorem]{Corollary}
\newtheorem{proposition*}[theorem]{Proposition}
\newtheorem{problem*}[theorem]{Problem}
\newtheorem{conjecture*}[theorem]{Conjecture}

\theoremstyle{defi}
\newtheorem{definition}[theorem]{Definition}
\newtheorem{hyp}{Assumption}

\newtheorem{problem}{Problem}

\theoremstyle{remexample}
\newtheorem{remark}[theorem]{Remark}

\newtheorem{teo}[theorem]{Theorem}
\newtheorem{lem}[theorem]{Lemma}
\newtheorem{pro}[theorem]{Proposition}
\newtheorem{cor}[theorem]{Corollary}

\theoremstyle{ass}

\begin{document}

\title[Mean-Field Games with Dirichlet conditions]{Existence of weak solutions to first-order stationary Mean-Field Games with Dirichlet conditions}

\author{Rita Ferreira}
\address[R. Ferreira]{
        King Abdullah University of Science and Technology (KAUST), CEMSE Division, Thuwal 23955-6900, Saudi Arabia.}
\email{rita.ferreira@kaust.edu.sa}
\author{Diogo Gomes}
\address[D. Gomes]{
        King Abdullah University of Science and Technology (KAUST), CEMSE Division, Thuwal 23955-6900, Saudi Arabia.}
\email{diogo.gomes@kaust.edu.sa}
\author{Teruo Tada}
\address[T. Tada]{
       King Abdullah University of Science and Technology (KAUST), CEMSE Division, Thuwal 23955-6900, Saudi Arabia.}
\email{teruo.tada@kaust.edu.sa}

\thanks{R. Ferreira, D. Gomes, and T. Tada were partially supported baseline and start-up funds from King Abdullah University of Science and Technology (KAUST). 
}
\date{\today}

\begin{abstract}
In this paper, we study first-order stationary monotone mean-field games (MFGs) with Dirichlet boundary conditions. While for Hamilton--Jacobi equations Dirichlet conditions may not be satisfied, here, we establish the existence of solutions of MFGs that satisfy those conditions. To construct these solutions, we introduce a monotone regularized problem. Applying Schaefer's fixed-point theorem and using the monotonicity of the MFG, we verify that there exists a unique weak solution to the regularized problem. Finally, we take the limit of the solutions of the regularized problem and using Minty's method, we show the existence of weak solutions to the original MFG.
\end{abstract}

\maketitle

\section{Introduction}

Mean-field games (MFGs) were introduced in the mathematical community in \cite{ll1}, \cite{ll2}, and \cite{ll3} and, independently, around the same time in
the engineering community in \cite{Caines2} and \cite{Caines1}. These games model the behavior of large populations of rational agents who seek to optimize an individual utility.
Here, we consider the following first-order stationary MFG with a Dirichlet boundary condition.
\begin{problem}\label{OP}
Suppose that $\Omega \subset \Rr^d$ is an open and bounded set with smooth boundary, $\partial \Omega$, and $\gamma>1$.  Let $V, \phi, h \in C^{\infty}(\overline{\Omega})$, $g\in C^\infty(\Rr_0^+)$, and $H\in C^\infty(\overline\Omega\times\Rr^d)$ be such that $g$ is increasing, $\phi\geq 0$ in $\Omega$, and $\int_{\Omega}\phi\, \dx=1$.
Find $(m,u)\in L^1(\Omega)\times W^{1,\gamma}(\Omega)$ satisfying $m\geq0$ and
\begin{equation*}
\begin{cases}
&-u  - H(x,Du) +g(m) - V(x)=0 \quad \rm{in}\  \Omega, \\
& m - \div\big(mD_pH(x,Du)\big) - \phi = 0 \quad \rm{in} \ \Omega, \\
& u=h\quad \rm{on}\ \partial\Omega. 
\end{cases}
\end{equation*}
\end{problem}
In Problem \ref{OP}, $H$ is the Hamiltonian and its Legendre transform, given by
\[
L(v)=\sup_p \{ -p\cdot v-H(x,p)\},
\]
gives the agent's cost of movement at speed  $v$;
  the potential, $V$, determines spatial preferences of each agent, and the coupling, $g$, encodes the interactions between agents and the mean field. 
  When agents leave the domain
  through a point $x\in \partial \Omega$, they incur a charge $h(x)$. 
 Finally, the source, $\phi$,
 represents an incoming flow of agents replacing the ones leaving through $\partial \Omega$ or due to the unit discount encoded in 
 the term $-u$ in the Hamilton--Jacobi equation and in the term $m$ in the transport equation.
  We note that $m$ is the density of the distribution of the agents and $u(x)$ the value function of an agent in the state $x$.

In general, Hamilton--Jacobi equations with Dirichlet boundary conditions do not admit continuous
solutions up to the boundary. This fact can be illustrated by the equation
\[
u'(x)=0
\]
for $0<x<1$ 
with $u(0)=0$ and $u(1)=1$. However, keeping the boundary conditions and coupling the previous equation with a transport equation
\[
\begin{cases}
u'(x)=m(x)\\
m'(x)=0,
\end{cases}
\]
we obtain a model that has a unique solution continuous up to the boundary,
$(u,m)=(x,1)$. This model motivates our main result, the existence of solutions for Problem \ref{OP} satisfying the Dirichlet boundary condition in the trace sense.

MFGs have been studied intensively, see, for example, the monograph \cite{GPV}, 
the surveys \cite{achdou2013finite}, \cite{GS},
the note \cite{Cardialaguet10}, and the lectures \cite{LCDF}. 
Because there are few known explicit solutions for stationary MFGs
\cite{Gomes2016b}, \cite{GNPr216}, and \cite{EvGomNur17}, a substantial effort has been undertaken
to develop numerical methods
\cite{almulla2017two} and \cite{Briceno-Arias}, find special transformations \cite{MR3377677}, 
and to establish the existence of solutions.  
The existence of solution for
second-order, stationary MFGs without congestion was investigated in \cite{GPM1}, \cite{GPatVrt}, \cite{PV15}, \cite{bocorsporr}, and \cite{2016arXiv161107187C}; problems with congestion were examined in 
\cite{EvGom}, \cite{GMit}. The theory for first-order MFGs is less developed. 
The existence of solutions for first or second-order stationary MFGs was examined in 
\cite{FG2} (also see \cite{almayouf2016existence}) using monotone operators and, using a variational approach, 
certain first-order MFGs with congestion were examined in 
\cite{EFGNV2017}.

For first-order MFGs and often for second-order MFGs, existing publications consider only 
periodic boundary conditions. However, in applications, MFGs with boundary conditions are quite natural: for example, Dirichlet boundary conditions arise in models where agents can
leave the domain and are charged an exit fee.   
Here, 
our goal is to prove the existence of weak solutions for the first-order stationary monotone MFG with Dirichlet boundary conditions.  
Thus, we introduce a notion of weak solutions to Problem~\ref{OP} similar to the one considered in \cite{FG2}. Here, however, we account for the Dirichlet boundary conditions. We recall that in \cite{FG2}, the authors consider stationary monotone MFGs with periodic boundary conditions, and their notion of weak solutions is induced by monotonicity. Monotonicity plays an essential role in 
the uniqueness of solutions \cite{LCDF}, and in its absence, the theory becomes substantially harder, 
see the examples in \cite{Gomes2016b}, and the non-monotone second-order MFGs in \cite{cirant2},  \cite{cirant3}, and \cite{FrS}. 

Throughout this paper, $k \in \Nn$ is a fixed natural number such that  $2k> \frac{d}{2}+2$, and   $\Aa$ and $H_h^{2k}(\Omega)  $ are the sets given, respectively, by
\begin{equation}\label{DAa}
\Aa:= \Big\{m \in H^{2k}(\Omega)\ |\ m\geq 0\Big\}
\end{equation}
and
\begin{equation}\label{DHh2k}
H_h^{2k}(\Omega):=\Big\{w\in H^{2k}(\Omega)\ |\ w-h\in H_0^{2k}(\Omega) \Big\}.
\end{equation}
%
%
\begin{definition}\label{DOPWS1}
A weak solution to Problem \ref{OP} is a pair $(m,u) \in L^1(\Omega) \times W^{1,\gamma}(\Omega)$ satisfying
\begin{flalign}
({\rm D1})\enspace & \, u=h \ \mbox{on }\partial\Omega,\enspace m \geq 0 \ \mbox{in }\Omega, &\nonumber \\ 
({\rm D2})\enspace & \, \left<
   F
  \begin{bmatrix}
      \eta \\
      v
  \end{bmatrix},
  \begin{bmatrix}
      \eta  \\
      v
  \end{bmatrix}
  -
  \begin{bmatrix}
      m  \\
      u
  \end{bmatrix}
\right> \geq 0
\quad \mbox{for all}\ (\eta, v)\in \Aa \times H_h^{2k}(\Omega),&\nonumber
\end{flalign}
where, for $(\eta, v) \in H^{2k}(\Omega) \times H^{2k}(\Omega)$ fixed, $F[\eta, v]: L^1(\Omega) \times L^1(\Omega) \to \Rr$ is the functional given by
\begin{equation}\label{DRF2}
\begin{aligned}
\left<
   F
  \begin{bmatrix}
      \eta \\
      v
  \end{bmatrix},
  \begin{bmatrix}
      w_1 \\
      w_2
  \end{bmatrix}
\right>
:=& 
\int_{\Omega}\big(-v-H(x,Dv) +g(\eta) -V\big)w_1 \, \dx
\\ &\quad+\int_{\Omega}\Big(\eta -\div\big(\eta D_pH(x,Dv)\big)-\phi\Big)w_2 \, \dx .
\end{aligned}
\end{equation}
\end{definition}
Next, we state our main theorem that, under the assumptions detailed in Section~\ref{Ass}, establishes the existence of weak solutions to Problem~\ref{OP}.
\begin{teo}\label{TOP}
Consider Problem \ref{OP} and suppose 
that Assumptions \ref{Hcoer}--\ref{Hmono} hold. Then, there exists a weak solution $(m,u)\in L^1(\Omega) \times W^{1,\gamma}(\Omega)$ to Problem \ref{OP} in the sense of Definition~\ref{DOPWS1}.
\end{teo}
In \cite{FG2}, the method of continuity was used to prove the existence of a weak solution to stationary monotone MFGs with periodic boundary conditions. Here, we use a different approach: we apply 
Schaeffer's fixed-point theorem and extend the results in \cite{FG2} to Dirichlet boundary conditions.

To prove Theorem~\ref{TOP}, we introduce a regularized problem, Problem~\ref{MP}, that we believe to be of interest on its own. This regularized problem preserves the monotonicity of the original MFG
in the sense of Assumption~\ref{Hmono}
(see Section~2 and Lemma~\ref{MORF}).
Note that the choice of boundary conditions is critical to preserve monotonicity. 
Moreover, because of the regularizing terms (see the $\epsilon$-terms in \eqref{HighdRP} below), it is simpler to prove the existence and uniqueness of weak solutions to this problem. Then, by letting $\epsilon\to 0$, we can construct a weak solution to Problem~\ref{OP} (see Section~\ref{PfMT2}).
\begin{problem}\label{MP}
Let $\Omega$ be an open and bounded set with smooth boundary, $\partial \Omega$, and outward pointing unit normal vector $\bf{n}$.
Let $ V, \phi, \xi, h \in C^{\infty}(\overline\Omega)$, $g\in C^\infty(\Rr_0^+)$, and $H\in C^\infty(\overline\Omega\times\Rr^d)$ be such that $g$ is increasing, $\phi \geq 0$ in $\Omega$, and $\int_{\Omega} \phi \, \dx=1$. Fix $\epsilon\in(0,1)$. Find $(m, u) \in H^{2k}(\Omega) \times H^{2k}(\Omega)$ satisfying $m\geq 0$ and
\begin{equation}\label{HighdRP}
\begin{cases}
&-u - H(x,Du) +g(m) -V(x)+ \epsilon \big(m + \Delta^{2k}m\big) = 0 \quad {\rm in}\  \Omega, \\
& m - \div\big(mD_pH(x,Du)\big) - \phi + \epsilon \big(u +\xi +\Delta^{2k}(u+\xi) \big)= 0 \quad {\rm in} \ \Omega, \\
& 
\frac{\partial}{\partial \bf{n}}\partial^{\alpha}\Delta^{i} m=0\ 
{\rm on}\  \partial \Omega\ 
{\rm for\ all}\ \alpha\in\Nn_0^d\ {\rm and}\ i \in\Nn_0\ {\rm such\ that}\ 
|\alpha|+i=2k-1,\\ 
&\partial^\beta u
=\partial^\beta h
\ \ 
{\rm on}\ \partial\Omega\ 
{\rm for\ all}\ \beta\in\Nn_0^d\ {\rm and}\ \mbox{such\ that}\ |\beta|\leq 2k-1.
\end{cases}
\end{equation}
\end{problem}
%
%
In the preceding problem, $\xi$ is a technical term used to cancel the boundary conditions in $u$ 
so that we can work with vanishing Dirichlet boundary conditions, see Section~\ref{PfMT1}.
Since the regularizing terms are of order greater than two, we cannot apply the maximum principle. Thus,  Problem~\ref{MP} may not have classical solutions with $m\geq 0$. Hence, in the following definition, we introduce a notion of weak solution to Problem \ref{MP} that requires positivity 
and relaxes the equality in the Hamilton--Jacobi equation. This definition is related to the ones in 
\cite{cgbt} and in \cite{FG2}, where $u$ is only required to be a subsolution of the Hamilton--Jacobi equation.  
%
\begin{definition}\label{DMPWS1}
A weak solution to Problem \ref{MP} is a pair $(m,u) \in H^{2k}(\Omega) \times H^{2k}(\Omega)$ satisfying, for all $w\in \Aa$ and $v \in H_0^{2k}(\Omega)$,
\begin{flalign*}
 ({\rm E1})\enspace & \,  u\in H_h^{2k}(\Omega),\ m\geq 0\ \mbox{in}\  \Omega, &\\
 \begin{split}  ({\rm E2})\enspace &  \int_\Omega \big(  -u - H(x,Du) +g(m) -  V \big)(w-m) \, \dx \\ 
&\qquad +\int_{\Omega} \Big[ \epsilon m(w-m)+\epsilon \sum_{|\alpha|=2k}\partial^\alpha
m(\partial^\alpha w-\partial^\alpha m)\Big]\, \dx \geq 0,\end{split} &\\
\begin{split}  ({\rm E3})\enspace&   \int_{\Omega} \big(m -\div\big(m D_p H(x,Du)\big)-\phi \big)v\,\dx  \\
&\qquad+\int_\Omega \Big[\epsilon \Big(uv + \sum_{|\alpha|=2k}\partial^\alpha u\partial^\alpha v\Big)
+\epsilon (\xi+\Delta^{2k} \xi) v \Big]\, \dx=0.
\end{split} &
\end{flalign*}
\end{definition}
%
%
\begin{remark}\label{rmkmEL}
Assume that $(m,u)$ is a weak solution to Problem~\ref{MP}. Let $\Omega' : = \{x \in \Omega\ |\  m(x)>0 \}$, and fix $w_1 \in C_c^\infty(\Omega')$. For all $\tau\in \Rr$ with $|\tau|$ sufficiently small, we have $w=m+\tau w_1\in \Aa$.
Then, from (E2), we obtain
\begin{equation*}
 \tau\int_\Omega \big(-u - H(x,Du) +g(m) - 
V \big)w_1 \, \dx+ \tau\int_{\Omega} \big( \epsilon mw_1+\epsilon \sum_{|\alpha|=2k}\partial^\alpha
m\partial^\alpha w_1\big)\, \dx \geq 0
.
\end{equation*}
Because the sign of $\tau$ is arbitrary, we conclude that $m$ satisfies
\begin{equation*}
- u  - H(x,Du) + g(m)-V(x) +\epsilon (m + \Delta^{2k}m) =0  \ \mbox{ pointwise\ in} \ \Omega'.
\end{equation*}
Moreover, let $w_2\in C_c^\infty(\Omega)$ be such that $w_2\geq 0$; taking $w=m+w_2\in \Aa$ in (E2) and integrating by parts, we obtain
\begin{equation*} 
\tau\int_\Omega \big(  -u - H(x,Du) +g(m) - 
V(x)  \big)w_2 \, \dx+\tau\int_{\Omega} \big( \epsilon mw_2+\epsilon \Delta^{2k}m w_2\big)\, \dx\geq 0
.
\end{equation*}
Thus, 
\begin{equation*}
- u  - H(x,Du) + g(m) - 
V(x)+\epsilon (m + \Delta^{2k}m) \geq 0  \ \mbox{  in the sense of distributions
in} \ \Omega.
\end{equation*}
Also, by (E3), we have
\begin{equation*}
m - \div\big(mD_pH(x,Du)\big) - \phi
+ \epsilon \big(u +\xi +\Delta^{2k}(u+\xi)
\big)= 0  \ \mbox{  in the sense of distributions
in} \ \Omega.
\end{equation*}
\end{remark}
Under the same assumptions of Theorem \ref{TOP}, we prove the existence and uniqueness of the weak solutions to Problem \ref{MP}.
\begin{teo}\label{MT}
Consider Problem \ref{MP} and suppose that Assumptions \ref{Hcoer}--\ref{AWC1} and \ref{Hmono} hold for some $\gamma>1$. Then, there exists a unique weak solution $(m,u) \in H^{2k}(\Omega)\times H^{2k}(\Omega)$ to Problem \ref{MP} in the sense of Definition~\ref{DMPWS1}.
\end{teo}

We begin by addressing Theorem~\ref{MT}. First, in Section~\ref{Prows}, we prove a priori estimates for classical and weak solutions to Problem~\ref{MP}. Second, in Sections~\ref{VP} and \ref{BF}, we introduce two auxiliary problems: a variational problem whose Euler--Lagrange equation is the first equation in Problem~\ref{MP} and a problem associated with a bilinear form corresponding to the second equation in that problem. In each of these two sections, we show that there exists a unique solution. Finally, in Section~\ref{PfMT1}, we prove Theorem~\ref{MT} using  Schaefer's fixed-point theorem together with the results established in Sections~\ref{Prows}--\ref{BF}. Finally, the proof
of 
Theorem~\ref{TOP} is given in Section \ref{PfMT2} using Minty's method. 


\section{Assumptions}\label{Ass}
To prove our main results, we need the following assumptions on the 
functions that arise in Problems \ref{OP} and \ref{MP}. 
The first three assumptions prescribe standard growth conditions on the Hamiltonian, $H$. For instance, \[H(x,p)=a(x)(1+|p|^2)^{\frac{\gamma}{2}}+ b(x)\cdot p,\] where $a\in C(\overline \Omega)$, $a(x)>0$ for all $x\in\overline\Omega$, and $b:\overline\Omega \to\Rr^{d}$ is a $C^{\infty}$ function satisfies our assumptions. 
\begin{hyp}\label{Hcoer}
There exists a constant, $C>0$, such that, for all $(x,p)\in \Omega\times\Rr^d$,
\begin{equation*}
\begin{aligned}
-H(x,p)+ D_pH(x,p)\cdot p\geq \frac{1}{C}|p|^\gamma -C.
\end{aligned}
\end{equation*}

\end{hyp}


\begin{hyp}\label{Hbdd}
There exists a constant, $C>0$, such that, for all $(x,p)\in \Omega\times\Rr^d$,
\begin{equation*}
\begin{aligned}
H(x,p)
\geq 
\frac{1}{C} |p|^\gamma -C.
\end{aligned}
\end{equation*}
\end{hyp}
\begin{hyp}\label{Bderi}
There exists a constant, $C>0$, such that, for all $(x,p)\in \Omega\times\Rr^d$,
\begin{equation*}
|D_pH(x,p)|\leq C|p|^{\gamma-1}+C .
\end{equation*}
\end{hyp}
The next three assumptions impose growth conditions on $g$. For example, these are satisfied for $g(m)=m^\alpha,\alpha>0$, or by $g(m)=\ln m$ (apart from small changes). 
\begin{hyp}\label{gint}
The function $g$ is increasing in $\Rr^+_0$. Moreover, for all $\delta>0$, there exists a constant, $C_\delta>0$, such that, for all $m \in L^1(\Omega)$,
\begin{equation*}
\max\bigg\{\int_\Omega |g(m)|\, \dx,
\int_\Omega m \,\dx \bigg\}
\leq
\delta \int_\Omega mg(m)\dx + C_\delta.
\end{equation*}
\end{hyp}
\begin{hyp}\label{AWC1}
There exists a constant, $C>0$, such that, for all $m\in L^1(\Omega)$,
\begin{equation*}
\int_\Omega mg(m)\, \dx\geq -C.
\end{equation*}
\end{hyp}
\begin{hyp}\label{gwc}
If $\{m_j\}_{j=1}^\infty\subseteq L^1(\Omega)$ is a sequence satisfying 
\begin{equation*}
\sup_{j\in\Nn} \int_\Omega m_jg(m_j)\, \dx <+\infty,
\end{equation*}
then there exists a subsequence of $\{m_{j}\}_{j=1}^\infty$ that converges weakly in $L^1(\Omega)$.
\end{hyp}
\begin{remark}
        If $g$ is an increasing function with $g\geq 0$ and $\lim_{t\to\infty}g(t)=\infty$, then $g$ satisfies Assumption~\ref{gwc}. This fact is a consequence of the {\it De la Vall\'ee Poussin} lemma together with the Dunford--Pettis theorem; see, for instance, Theorems~$2.29$ and~$2.54$ in \cite{FoLe07}.
\end{remark}
%
%
Our final assumption concerns the monotonicity of the functional $F$ introduced in Definition~\ref{DOPWS1}. Monotonicity is the key ingredient in the proof of Theorem \ref{TOP} from 
Theorem \ref{MT} through Minty's method. 
\begin{hyp}\label{Hmono}
The functional $F$ introduced in Definition~\ref{DOPWS1} is monotone with respect to the $L^2\times L^2$-inner product; that is,
for all $(\eta_1, v_1)$, $(\eta_2, v_2) \in \Aa   \times H_h^2(\Omega)$, $F$ satisfies
\begin{equation*}
\left<
   F
  \begin{bmatrix}
      \eta_1  \\
      v_1 
  \end{bmatrix}
  -
   F
  \begin{bmatrix}
      \eta_2  \\
      v_2 
  \end{bmatrix},
  \begin{bmatrix}
      \eta_1  \\
      v_1
  \end{bmatrix}
  -
  \begin{bmatrix}
      \eta_2  \\
      v_2
  \end{bmatrix}
\right>
\geq 0.
\end{equation*}
\end{hyp}
%
%
%
%
%

%
%
\section{Properties of weak solutions}\label{Prows}
In this section, we investigate properties of weak solutions, $(m,u)$, to Problem~\ref{MP}. First, we prove an a priori estimate for classical solutions. Then, we verify that this a priori estimate also holds for weak solutions. Finally, we show that $u$ is bounded in $W^{1,\gamma}(\Omega)$,
and that \( (\sqrt{\epsi} m,  \sqrt{\epsi} u)\) is bounded in
\(H^{2k}(\Omega)\times H^{2k}(\Omega)\). 
In Section~\ref{PfMT1}, we combine
these estimates with  Schaefer's fixed-point theorem to prove the existence of weak solutions to Problem~\ref{MP}.

To simplify the notation, throughout this section, we use the same letter $C$ to denote any positive \textit{constant that depends only on the problem data}; that is, may depend
on \(\Omega\), \(d\), \(\gamma\), $H,V,\phi,\xi$, and $h$,   on the constants in the Assumptions~\ref{Hcoer}--\ref{AWC1}, or
on universal constants such as the constant in Morrey's theorem or  Gagliardo--Nirenberg
interpolation
inequality. 
In particular, these constants do not depend on the particular choice of solutions to Problem~\ref{MP} nor on \(\epsilon\).
\begin{pro}\label{apriori1}
Consider Problem \ref{MP} and suppose that Assumptions~\ref{Hcoer}--\ref{gint} hold for some \(\gamma>1\). Then, there exists a positive constant, $C$, that depends only on the
problem data  such that for any solution $(m, u)$ to Problem~\ref{MP} in the classical sense, we have
\begin{equation}\label{aprimu}
\int_{\Omega} \Big(
mg(m)+\frac{1}{C}m|Du|^\gamma 
+
\frac{1}{C}\phi|Du|^\gamma \Big)\, \dx
+
{\epsilon} \int_{\Omega} \
\Big[
m^2+u^2+\sum_{|\alpha|=2k}\big((\partial^\alpha m)^2+(\partial^\alpha u)^2\big)
 \Big] \, \dx
\leq C. 
\end{equation}
\end{pro}
%
%
\begin{proof}

Multiplying the first equation in \eqref{HighdRP} by $(m-\phi)$ and the second one by $(u-h)$, 
adding and integrating over $\Omega$, and then using integration by parts and the boundary conditions, we have
\begin{equation}\label{E3.2}
 \begin{split}
  \int_{\Omega} \bigg[
  &m g(m)+m \big(-H(x,Du)+ D_p H(x,Du)\cdot Du\big)+\phi H(x,Du) \\
  &\qquad\qquad\qquad\qquad\qquad\qquad
  +\epsilon \Big(m^2+u^2+\sum_{|\alpha|=2k}\big((\partial^\alpha m)^2+(\partial^\alpha u)^2\big)\Big)
  \bigg] \, \dx
  \\
  = \int_\Omega\bigg[
  &\phi g(m)
  +(\epsilon \phi + V +h) m + ( \epsilon \xi h -V\phi - \phi h) + m D_pH(x,Du)\cdot Dh\\
   &\qquad\quad 
   + \epsilon \Big( u(h-\xi) + \sum_{|\alpha|=2k} \big(\partial^\alpha \phi \partial^\alpha m +  \partial^\alpha u( \partial^\alpha h - \partial^\alpha \xi) +\partial^\alpha \xi\partial^\alpha h\big) \Big) \bigg]\, \dx.
  \end{split}
\end{equation}
By Assumptions~\ref{Hcoer}--\ref{Bderi}, Young's inequality, and the positivity of $m$ and $\phi$, we have
\begin{equation*}
\begin{aligned}
&\int_\Omega m\big(-H(x.Du)+D_pH(x,Du)\cdot Du\big)\, \dx 
\geq 
\int_\Omega\Big( \frac{m|Du|^\gamma}{C} -Cm\Big)\, \dx,\\
&\int_\Omega \phi H(x,Du)\,\dx \geq \int_\Omega \bigg(\frac{\phi |Du|^\gamma}{C}-C\phi \bigg)\,\dx,\\
&\int_\Omega mD_pH(x,Du)\cdot Dh\,\dx
\leq
\int_\Omega Cm(|Du|^{\gamma-1}+1)
\leq
\int_\Omega\bigg( \frac{m|Du|^\gamma}{2C} +C m\bigg) \,\dx.
\end{aligned}
\end{equation*}
Therefore, from Assumption~\ref{gint}, Young's inequality, and \eqref{E3.2}, we obtain
\begin{equation*}
 \begin{split}
   \int_{\Omega} &\bigg[ mg(m) + \frac{m|Du|^\gamma}{C} +  \frac{\phi |Du|^\gamma}{C} 
   + \frac{\epsilon}{2} \Big( m^2 + u^2 + \sum_{|\alpha|=2k}\big((\partial^\alpha m)^2+(\partial^\alpha u)^2\big)\Big)\bigg] \, \dx \\
   \leq 
   \int_{\Omega}&\Big( 
   \phi g(m)+ Cm +  \frac{m|Du|^\gamma}{2C} \Big)\, \dx + C
   \leq 
   \frac{1}{2} \int_{\Omega}\Big( mg(m) + \frac{m|Du|^\gamma}{C} \Big) \, \dx + C,
 \end{split}
\end{equation*}
from which Proposition~\ref{apriori1} follows.
\end{proof}

\begin{cor}\label{apriWS1}
Consider Problem \ref{MP} and suppose that Assumptions~\ref{Hcoer}--\ref{gint} hold for some \(\gamma>1\), and let $(m, u)$ be a weak solution to Problem~\ref{MP} in the sense of Definition~\ref{DMPWS1}. Then, $(m, u)$ satisfies the estimate \eqref{aprimu}. 
\end{cor}
\begin{proof}
Let
$(m, u)$ be a weak solution to Problem~\ref{MP} in the sense of Definition~\ref{DMPWS1}. Taking $v = u- h \in H_0^{2k}(\Omega)$ and $w= \phi \in \Aa$ in  (E2) and  (E3) in Definition~\ref{DMPWS1}, and then summing the resulting inequalities, we obtain \eqref{E3.2} with ``$=$'' replaced by ``$\leq$". Thus, arguing as in Proposition~\ref{apriori1}, we conclude that $(m, u)$ satisfies  \eqref{aprimu}.
\end{proof}

\begin{cor}\label{apriDu1}
Consider Problem \ref{MP} and suppose that Assumptions~\ref{Hcoer}--\ref{AWC1} hold for some \(\gamma>1\). Then, there exists a positive constant, $C$, that depends only on the
problem data such that for any weak solution $(m,u)$ to Problem~\ref{MP}, we have \(\Vert u\Vert_{W^{1,\gamma}(\Omega)}\leq C\).
\end{cor}
\begin{proof}
By Corollary~\ref{apriWS1}, the positivity
of \(m\) and \(\phi\), and Assumption~\ref{AWC1}, we have
\begin{equation*}
\begin{aligned}
\bigg|\int_\Omega m g(m)\,\dx \bigg| \leq C.
\end{aligned}
\end{equation*}

On the other hand, using (E2) in Definition~\ref{DMPWS1}
with   \(w:=m
+1\in \Aa\), \(\epsilon\leq 1\), Assumption~\ref{gint}, the previous estimate, and a weighted Young's inequality,  we obtain
\begin{equation*}
\int_{\Omega} H(x,Du)\, \dx \leq \int_{\Omega}\big( \epsilon m - u +g(m) -V\big)\, \dx
\leq
\sigma\int_{\Omega} |u|^\gamma \, \dx + C _\sigma,
\end{equation*}
where $\sigma>0$ is arbitrary. Moreover, because $u-h=0$ on $\partial \Omega$,  Poincar\'e's inequality yields
 \begin{align}\label{boundonu}
\int_{\Omega} |u|^\gamma \, \dx\leq C 
\int_{\Omega} | Du|^\gamma \, \dx+C.
\end{align}
Then, invoking  Assumption~\ref{Hbdd} and using the estimates above with \(\sigma\) chosen appropriately, we obtain
\begin{equation*}
\int_{\Omega} |Du|^\gamma \, \dx \leq C.
\end{equation*}
Using \eqref{boundonu} once more, we conclude that  \(\Vert u\Vert_{W^{1,\gamma}(\Omega)}\leq C\), where \(C\) is a positive constant that depends only on the
problem data.
 \end{proof}

\begin{cor}\label{aprisqrt}
Consider Problem \ref{MP} and suppose that
Assumptions~\ref{Hcoer}--\ref{AWC1} hold for some \(\gamma>1\). Then, there exists a positive constant,
$C$, that depends only on the
problem data such that for any weak solution $(m,u)$ to Problem~\ref{MP},
we have \(\Vert \sqrt{\epsi} m\Vert_{H^{2k}(\Omega)} + \Vert \sqrt{\epsi} u\Vert_{H^{2k}(\Omega)}\leq C\).
\end{cor}
\begin{proof}
Using  Corollary~\ref{apriWS1},
Assumption~\ref{AWC1}, and the positivity of  \(m\) and \(\phi\), we obtain
\begin{equation*}
\begin{aligned}
{\epsilon} \int_{\Omega} \
\Big[
m^2+u^2+\sum_{|\alpha|=2k}\big((\partial^\alpha m)^2+(\partial^\alpha u)^2\big)
 \Big] \, \dx
\leq C,
\end{aligned}
\end{equation*}
where \(C\) is a positive constant  that depends only on the
problem data. The conclusion follows from the  Gagliardo--Nirenberg
interpolation
inequality.
\end{proof}

\section{An auxiliary variational problem}\label{VP}
Here, we introduce a variational problem whose Euler--Lagrange equation is the first equation in \eqref{HighdRP}. We prove the existence and uniqueness of a minimizer, $m$, to this variational problem. Then, we investigate properties of $m$ that enable us to prove the existence and uniqueness of a weak solution to Problem~\ref{MP} in Section~\ref{PfMT1}.

Given $(m, u) \in H^{2k-2}(\Omega) \times H^{2k-1}(\Omega)$ with $m \geq 0$, let $I_{(m, u)}: H^{2k}(\Omega) \to \Rr$ be the functional defined, for $w\in H^{2k}(\Omega)$, by
\begin{equation}\label{defImu}
I_{(m, u)}[w]: =  \int_{\Omega} \Big[ \frac{\epsilon }{2} \Big( w^2 + \sum_{|\alpha|=2k}(\partial^\alpha w)^2\Big) + \big( - u  - H(x,Du)+ g(m)-V \big) w \Big] \, \dx.
\end{equation}
Note that by Morrey's embedding theorem,
\(H^{2k-2}(\Omega)\) is compactly embedded
in \(C^{0,l}(\overline \Omega)\) for
some \(l\in (0,1)\). In particular,
there exists a positive constant, \(C=C(\Omega,k,d,l)\),
such that
for all \(\vartheta\in H^{2k-2}(\Omega)\),
we have
\begin{equation}
\label{eq:MorreyET}
\begin{aligned}
\Vert \vartheta \Vert_{C^{0,l}(\overline \Omega)} \leq C \Vert \vartheta \Vert_{H^{2k-2}(
\Omega)}.
\end{aligned}
\end{equation}
Next, we fix $(m_0,u_0)\in H^{2k-2}(\Omega)\times H^{2k-1}(\Omega)$ with $m_0\geq0$, and set $I_0=I_{(m_0,u_0)}$. We address the variational problem finding $m\in \Aa$ such that 
%
\begin{equation}\label{VP1}
I_0[m] =\inf_{w \in \Aa}I_0[w],
\end{equation}
where $\Aa$ is defined in \eqref{DAa}.
%
%
\begin{pro}\label{EMVP1} 
Let $H, g,$ and $V$ be as in Problem~\ref{MP}, and 
fix $(m_0, u_0) \in H^{2k-2}(\Omega) \times H^{2k-1}(\Omega)$ with $m_0 \geq 0$. Then, there exists a unique $m \in \Aa$ satisfying \eqref{VP1}.
\end{pro}
%
%
\begin{proof}  
Let $\{m_{n}\}_{n=1}^{\infty} \subset \Aa$ be a minimizing sequence for \eqref{VP1}, and fix $\delta\in(0,1)$. 
Then, there exists $N\in \Nn$ 
such that for all $n \geq N$, 
\begin{equation}\label{Bvp1}
I_0[m_n]<\inf_{w\in\Aa} I_0[w] +\delta \leq I_0[0] + 1=1. 
\end{equation}
By \eqref{eq:MorreyET}, we have 
$m_0\in C^{0,l}(\overline\Omega)$ and $u_0 \in C^{1,l }(\overline\Omega)$ for some $l\in(0,1)$, and  
$C_0:=\|- u_0 - H(x,Du_0) + g(m_0) -V\|_{L^\infty(\Omega)} \in \Rr$.
Then, using Young's inequality and \eqref{Bvp1}, for all $n\geq N$, we have
\begin{equation}\label{eqvp-3}
\frac{\epsilon }{2}\int_{\Omega} \Big[ m_{n}^2 + \sum_{|\alpha|=2k}(\partial^\alpha m_n)^2\Big] \, \dx \leq \int_{\Omega}C_0 {m_n}  \, \dx + 1
\leq 
\frac{\epsilon }{4} \int_{\Omega} m_n^2  \, \dx +\frac{C_0^2}{\epsilon}+1.
\end{equation}
By the Gagliardo--Nirenberg interpolation inequality, we have 
\begin{equation}\label{eqvp-4}
\|\partial^\alpha m_n\|_{L^2(\Omega)}^2 
\leq 
C \big(\|m_n\|_{L^2(\Omega)}^2 +\|D^{2k}m_n\|_{L^2(\Omega)}^2\big),
\end{equation}
where $\alpha$ is any multi-index such that $|\alpha|\leq 2k$.
Hence, from \eqref{eqvp-3} and \eqref{eqvp-4}, we conclude that $\{m_{n}\}_{n=1}^{\infty}$ is bounded in $H^{2k}(\Omega)$.
Consequently, $m_{n} \rightharpoonup m$ weakly in $H^{2k}(\Omega)$ for some $m\in H^{2k}(\Omega)$, extracting a subsequence if necessary. 
Because $m_{n} \geq 0$, also $m\geq0$; so, $m \in \Aa$. 

Moreover, $m_n\to m$ in $L^2(\Omega)$ and $\Vert D^{2k} m\Vert_{L^2(\Omega)}^2 \leq \liminf_{n \to \infty} \Vert D^{2k} m_n \Vert_{L^2(\Omega)}^2$; hence, $I_0[m]\leq \liminf_n I_0[m_n]=\inf I_0[w]\leq I_0[m]$, which shows that $m$ is a minimizer of $I_0$ over $\Aa$.

We now prove uniqueness. Assume that $m, \widetilde{m} \in \Aa$ are minimizers of $I_0$ over $\Aa$
 with $m \neq \widetilde{m}$. Then, $\frac{m + \widetilde{m}}{2} \in \Aa$ and,
recalling that $m-\widetilde{m}\in C^0(\overline\Omega)$, $\int_{\Omega}(m-\widetilde{m})^2\, \dx>0$. Moreover,
\begin{equation}\label{eqvp-5}
 \begin{aligned}
 I_0\left[\frac{m + \widetilde{m}}{2}\right]
 &= \int_{\Omega} \bigg[ \frac{\epsilon }{2}
  \Big( \Big(\frac{m+\widetilde{m}}{2}\Big)^2 
  + \sum_{|\alpha|=2k}\Big(\frac{\partial^\alpha m+\partial^\alpha \widetilde{m}}{2}\Big)^2\Big)\\
   &\qquad + \big( - u_0 - H(x,Du_0) + g(m_0) - V \big)
    \Big(\frac{m+\widetilde{m}}{2}\Big) \bigg]\, \dx \\
& = \frac{1}{2}\int_{\Omega} \bigg[
     \frac{\epsilon }{2} \Big( m^2 + \sum_{|\alpha|=2k}(\partial^\alpha m)^2 \Big) 
    + \big( - u_0 - H(x,Du_0) + g(m_0) - V \big)m \bigg] \, \dx \\
   &\qquad+\frac{1}{2}\int_{\Omega} \bigg[
    \frac{\epsilon }{2} \Big( \widetilde{m}^2 + \sum_{|\alpha|=2k}(\partial^\alpha \widetilde{m})^2 \Big)
    + \big( - u_0 - H(x,Du_0) + g(m_0) - V\big)\widetilde{m} \bigg] \, \dx \\
  &\qquad- \frac{\epsilon}{8}\int_{\Omega} \Big[ (m - \widetilde{m})^2  + \sum_{|\alpha|=2k}(\partial^\alpha m - \partial^\alpha \widetilde{m})^2  \Big] \, \dx \\
 &<\frac{1}{2}I_0[m] + \frac{1}{2}I_0[\widetilde{m}] = \min_{w \in \Aa}I_0[w],
 \end{aligned}
\end{equation}
which  contradicts the fact that $\frac{m+\widetilde{m}}{2} \in \Aa$. Thus, $m=\widetilde{m}$.
\end{proof}

\begin{cor}\label{aprimvp1}
Let $H,g,$ and $V$ be as in Problem~\ref{MP}, fix  $(m_0, u_0) \in H^{2k-2}(\Omega) \times H^{2k-1}(\Omega)$ with $m_0 \geq 0$, and let $m \in \Aa$ be the unique solution to \eqref{VP1}. Set $C_0:=\|u_0-H(x,Du_0)+g(m_0)-V\|_{L^\infty(\Omega)}$. Then, there exists a positive constant, $C$, depending only on the problem data
and on \(C_0\), such that
 $\Vert m\Vert_{H^{2k}(\Omega)} \leq C $.
\end{cor}
\begin{proof}
Because $I_0[m]\leq I_0[0],$ \eqref{eqvp-3} and \eqref{eqvp-4} hold with $m_n$ replaced by $m$,
from which Corollary~\ref{aprimvp1} follows.
\end{proof}

\begin{pro}\label{PVI}
Let $H, g,$ and $V$ be as in Problem \ref{MP}, fix
$(m_0, u_0) \in H^{2k-2}(\Omega) \times H^{2k-1}(\Omega)$ with $m_0 \geq 0$, and let $m \in \Aa$ be the unique solution to \eqref{VP1}. Then, for all $w \in \Aa$, $m$ satisfies
\begin{equation}\label{VI1}
\begin{aligned}
&\int_{\Omega} \big( -u_0 - H(x,Du_0) + g(m_0) -V\big)( w - m )  \, \dx \\
&\qquad+\int_{\Omega}\Big[\epsilon m ( w - m ) 
+ \epsilon \sum_{|\alpha|=2k}\partial^\alpha m(\partial^\alpha w-\partial^\alpha m)\Big] \, \dx 
\geq 0.
\end{aligned}
\end{equation}
\end{pro}
%
%
\begin{proof} 
Let $w \in \Aa$. If $\tau \in [0,1]$, then $m + \tau(w - m)=( 1 - \tau)m + \tau w \in \Aa$;
hence, the mapping $i: [0,1] \to \Rr$ given by
\begin{equation*}
i[\tau] := I_0 \big[m + \tau(w - m)\big]
\end{equation*}
is a well-defined $C^1$-function. 

Because $i(0)\leq i(\tau)$ for all $0\leq \tau \leq 1$, we have $i'(0) \geq 0$. 
On the other hand, for $0<\tau\leq1$,
we have\begin{equation*}
 \begin{aligned}
\frac{1}{\tau}\big(i(\tau)-i(0)\big)
 =&\int_{\Omega} \big ( - u_0 - H(x,Du_0) + g(m_0) -V\big)( w - m ) \, \dx \\
&+ {\epsilon} \int_{\Omega} \Big[
m( w - m) +\sum_{|\alpha|=2k}\partial^\alpha
m(\partial^\alpha w-\partial^\alpha
m)\Big] \, \dx\\
&+\tau\frac{\epsilon}{2} \int_{\Omega} \Big[( w - m)^2+  \sum_{|\alpha|=2k}(\partial^\alpha w -\partial^\alpha m )^2\Big] \, \dx.
 \end{aligned}
\end{equation*}
Consequently, letting $\tau\to 0^+$ in this equality and using $i'(0)\geq 0$, we obtain \eqref{VI1}.
\end{proof}

%
%

\begin{pro}\label{ELEI}
Let $H, g,$ and $V$ be as in Problem~\ref{MP}, fix $(m_0, u_0) \in H^{2k-2}(\Omega) \times H^{2k-1}(\Omega)$ with $m_0 \geq 0$, and let $m$ be the unique solution of \eqref{VP1}. Set $\Omega_1 = \{x \in \Omega\ |\  m(x)>0 \}$;
then, $m$ satisfies
\begin{equation*}\label{VWE1}
- u_0 - H(x,Du_0) + g(m_0)  - V +\epsilon ( m + \Delta^{2k}m )= 0 \quad \mbox{pointwise\ in} \ \Omega_1 \\
\end{equation*}
and
\begin{equation*}
- u_0 - H(x,Du_0) + g(m_0)  -V+\epsilon ( m + \Delta^{2k}m )\geq 0 \quad \mbox{in the sense of distributions in} \ \Omega. 
\end{equation*}
\end{pro}
\begin{proof}
To prove Proposition~\ref{ELEI}, it suffices to argue as in Remark~\ref{rmkmEL}, invoking \eqref{VI1} in place of (E2) and recalling the embedding $H^{2k-2}(\Omega)\hookrightarrow C^{0,l}(\overline\Omega)$ for some $l\in(0,1)$.
\end{proof}


\section{
A problem given by a bilinear form}\label{BF}
Here, we consider an auxiliary problem determined by a bilinear form related to the second equation in \eqref{HighdRP}.
Using  the Lax--Milgram Theorem, we prove the existence and uniqueness of a solution, $u$, to this auxiliary problem, and we establish a uniform bound on \(u\). These results are used in Section~\ref{PfMT1} to study the existence and uniqueness of a weak solution to Problem~\ref{MP}.

With $H,\phi,$ and $\xi$  as in Problem~\ref{MP}, 
given $(m, u) \in H^{2k-2}(\Omega) \times H^{2k-1}(\Omega)$ with $m \geq 0$, we define a bilinear form, $B:H_0^{2k}(\Omega) \times H_0^{2k}(\Omega) \to \Rr$, and a linear functional, $f_{(m, u)}:L^2(\Omega) \to \Rr$, by setting, for $v_1,v_2\in H_0^{2k}(\Omega)$ and $v\in L^2(\Omega)$,
\begin{equation}
\label{defBfmu}
\begin{aligned}
&B[v_1,v_2]:= \int_{\Omega} \epsilon \Big(v_1v_2 + \sum_{|\alpha|=2k}\partial^\alpha v_1 \partial^\alpha v_2\Big ) \, \dx,\\
&\big\langle f_{(m, u)},v\big\rangle
:=\int_{\Omega} \big[-m +\div\big(m D_pH(x,Du)\big)+\phi - 
\epsilon (\xi+\Delta^{2k}\xi)\big]v \, \dx.
\end{aligned}
\end{equation}
Fix $(m_0,u_0)\in H^{2k-2}(\Omega) \times H^{2k-1}(\Omega)$ with \(m_0\geq0\), and set  $f_0:= f_{(m_0,u_0)}$.  We address next the problem of finding $u\in H_0^{2k}(\Omega)$ satisfying\begin{equation}\label{BL1}
B[u,v]=\langle f_0,v \rangle \quad \mbox{for all}\ v \in H_0^{2k}(\Omega).
\end{equation}

\begin{pro}\label{ESBP1}
Let $H,\phi,$ and $\xi$ be as in Problem~\ref{MP}, and fix $(m_0, u_0) \in H^{2k-2}(\Omega) \times H^{2k-1}(\Omega)$ with $m_0 \geq 0$. Then, there exists a unique  solution, $u \in H_0^{2k}(\Omega)$, to \eqref{BL1}.
\end{pro}
\begin{proof}
Because  $(m_0,u_0)\in \big(H^{2k-2}(\Omega) \times H^{2k-1}(\Omega)\big) \cap \big( C^{0,l}(\overline\Omega) \times 
C^{1,l}(\overline\Omega)\big)$ for some $l\in(0,1)$ (see \eqref{eq:MorreyET}),
we have   $(-m_0 + \div\big(m_0D_p H(x,Du_0)\big)+\phi -\epsilon (\xi+\Delta^{2k}\xi) ) \in L^2(\Omega)$. Hence, by H\"older's inequality, $f_0$ is bounded in \(L^2(\Omega)\). 

Using H\"older's and Poincar\'e's inequalities, we have $|B[v_1,v_2]| \allowbreak \leq  c \Vert v_1\Vert_{H_0^{2k}(\Omega)}\Vert v_2\Vert_{H_0^{2k}(\Omega)}$ for all   $v_1,v_2\in H_0^{2k}(\Omega)$, where $c>0$ is a constant  independent of $v_1$ and $v_2$. 
Moreover,  we clearly have $B[v_1,v_1]\geq\epsi \Vert v_1\Vert_{H^{2k}_0(\Omega)}^2$ for all   $v_1\in H_0^{2k}(\Omega)$. 

Therefore, by the Lax--Milgram Theorem, there exists a unique  $u \in H_0^{2k}(\Omega)$ satisfying \eqref{BL1}. 
\end{proof}


\begin{lem}\label{ApriBL}
Let $H,\phi,$ and $\xi$ be as in Problem~\ref{MP}, fix $(m_0, u_0) \in H^{2k-2}(\Omega) \times H^{2k-1}(\Omega)$ with $m_0 \geq 0$, and let $u$ be the unique  solution to \eqref{BL1} in $H_0^{2k}(\Omega)$.  Then, there exists a positive constant, $C$,
depending only on the problem data and on \( \Vert m_0\Vert_{H^{2k-2}(\Omega)}\) and  
\(\Vert u_0\Vert_{H^{2k-1}(\Omega)}\),
such that
$\Vert u\Vert_{H^{2k}(\Omega)} \leq C$.
\end{lem}
\begin{proof}
Arguing  as in the proof of the Proposition~\ref{ESBP1}, we have $c_0:=\Vert -m_0 +
\div\big(m_0D_p H(x,Du_0)\big)+\phi -\epsilon (\xi+\Delta^{2k}\xi) \Vert_{ L^2(\Omega)}^2 \in \Rr^+_0$.
Then, using Young's 
inequality, we obtain
\begin{equation*}
\epsilon \left(\Vert u\Vert_{L^2(\Omega)}^2 +\Vert u\Vert_{H^{2k}_0(\Omega)}^2\right)
=
B[u,u] = \langle f_0, u \rangle
\leq
\epsilon\Vert u\Vert_{L^2(\Omega)}^2 + \frac{c_0}{4\epsilon}  .
\end{equation*}
Hence,  $\Vert u\Vert_{H^{2k}_0(\Omega)}^2
 \leq {c_0}/({4\epsilon^2})$, and the conclusion follows by Poincar\'e's inequality.
\end{proof}


\section{Proof of Theorem~\ref{MT}}\label{PfMT1}
This section is devoted to the proof of Theorem \ref{MT}. 
First, using  Schaefer's fixed-point theorem, we prove existence and uniqueness of a weak solution to \eqref{HighdRP} with $h\equiv 0$. Next, we apply this result to address the case of an arbitrary $h\in C^{\infty}(\overline\Omega)$. 

Let $\widetilde\Aa$ be the subset of \(\Aa\) (see  \eqref{DAa}) given by
\begin{equation*}
\widetilde\Aa:=\{w \in H^{2k-2}(\Omega)\ |\ w\geq 0 \},
\end{equation*}
 and consider the mapping $A: \widetilde\Aa \times H^{2k-1}(\Omega) \to \widetilde\Aa \times H^{2k-1}(\Omega)$ defined, for   $(m_0, u_0) \in \tilde A
\times H^{2k-1}(\Omega)$,  by
\begin{equation}\label{OpeA}
       A
    \begin{bmatrix}
      m_0  \\
      u_0        
    \end{bmatrix}
  :=
    \begin{bmatrix}
   m_0^* \\
   u_0^*
    \end{bmatrix},
\end{equation}
where $m_0^* \in \Aa$ is the unique solution to \eqref{VP1} and $u_0^*\in H_0^{2k}(\Omega) $ is the unique solution to \eqref{BL1}.
 In the following proposition, we  show  that  $A$ is  continuous and compact.


\begin{pro}\label{ACC} 
Let $H,g, V, \phi,$ and $\xi$ be as in Problem~\ref{MP}. Then,  the mapping $A: \widetilde\Aa \times H^{2k-1}(\Omega) \to \widetilde\Aa \times H^{2k-1}(\Omega)$ defined by \eqref{OpeA} is continuous and compact.
\end{pro}
%
%
\begin{proof}

We start by proving that \(A\) is continuous.
Let $(m_0, u_0),\, ( m_n, u_n) \in \widetilde\Aa  \times H^{2k-1}(\Omega) $ be such that 
$m_n \to m_0$ in $H^{2k-2}(\Omega)$ and $u_n \to u_0$ in $H^{2k-1}(\Omega)$. We want to show  that 
$m_n^* \to m_0^*$ in $H^{2k-2}(\Omega)$ and $u_n^* \to u_0^*$ in $H^{2k-1}(\Omega)$, where 
\begin{equation*}
\begin{aligned}
        \begin{bmatrix}
         m_0^* \\
         u_0^*
        \end{bmatrix}
= A \begin{bmatrix}
      m_0  \\
      u_0        
    \end{bmatrix}
 \enspace \text{and} \enspace 
 \begin{bmatrix}
         m_n^* \\
         u_n^*
        \end{bmatrix}
= A \begin{bmatrix}
      m_n  \\
      u_n        
    \end{bmatrix}. 
\end{aligned}
\end{equation*}

Recalling \eqref{defImu} and \eqref{defBfmu}, we set  $I_n:=I_{(m_n, u_n)}$ and $f_n := f_{(m_n, u_n)}$. By the definition of \(A\), we have that 
$ (m_0^*,  u_0^*)$ and $ (m_n^*, u_n^*)$ belong to $  \Aa  \times H_0^{2k}(\Omega) $
 and satisfy, for all \(v\in L^2(\Omega)\),
\begin{equation*}
\begin{aligned}
I_0[m_0^*]=\min_{w \in \Aa} I_0[w],\enspace I_n[m_n^*]=\min_{w \in \Aa} I_n[w],
\enspace
B[u_0^*,v] = \langle f_0,v \rangle,\enspace B[u_n^*,v] = \langle f_n,v \rangle.
\end{aligned}
\end{equation*}
Also, using  \eqref{eq:MorreyET}, there exists a positive constant,
\(c>0\), independent of \(n\in\Nn\), such that 
\begin{equation}
\label{eq:bounds}
\begin{aligned}
\sup_{n\in\Nn} 
(\Vert m_0\Vert_{L^{\infty}(\Omega)} + \Vert m_n\Vert_{L^{\infty}(\Omega)} +\Vert u_0\Vert_{W^{1,\infty}(\Omega)} + \Vert u_n\Vert_{W^{1,\infty}(\Omega)})
< c.
\end{aligned}
\end{equation}
Then, because \(H\), \(D_pH\), and \(g\) are locally Lipschitz functions, we have
\begin{equation}
\label{eq:Lips}
\begin{aligned}
\tilde c:=\max\big \{Lip(H; \Omega\times B(0,c)), Lip(D_pH; \Omega\times B(0,c)),Lip(g;B(0,c))\big\}\in \Rr^+_0. 
\end{aligned}
\end{equation}

Using the fact that $m_n^*$ and $m_0^*$ are minimizers, we have
\begin{equation*}
I_0[m_0^*]+I_n[m_n^*] \leq I_0\left[\frac{m_0^* + m_n^*}{2}\right] + I_n\left[\frac{m_0^*
+ m_n^*}{2}\right].
\end{equation*}
Then,  exploiting the  second equality in \eqref{eqvp-5} in
the preceding estimate first, and then   using Young's inequality and  \eqref{eq:Lips}, we obtain 
\begin{equation}\label{eq1-6-1}
 \begin{split}
  &\int_{\Omega}\frac{\epsilon}{4}\Big[ \big(m_0^*-m_n^*\big)^2 + \sum_{|\alpha|=2k}\big(\partial^\alpha m_{0}^* - \partial^\alpha m_{n}^* \big)^2 \Big] \, \dx \\
  \leq
  &\int_{\Omega}\frac{1}{2}|m_0^* - m_n^*|\left( | u_0 - u_n |  
  +  \big| H(x,Du_0)-H(x,Du_n)\big|+ |g( m_0)
- g(m_n )| \right) \, \dx \\
  \leq
  &\int_{\Omega}\left[\frac{\epsilon}{8}(m_0^* - m_n^*)^2 + \frac{6}{\epsilon}\left( ( u_0 - u_n )^2
     +  \tilde c^2|Du_{0} - Du_{n}|^2+ \tilde c^2( m_0 - m_n )^2\right) \right]\, \dx.
\end{split}
\end{equation}
Because  
$m_n \to m_0$ in $H^{2k-2}(\Omega)$ and $u_n \to u_0$ in $H^{2k-1}(\Omega)$,   \eqref{eq1-6-1} yields 
\begin{equation*}
\lim_{n\to\infty}\Vert m_0^* - m_n^* \Vert_{L^2(\Omega)} = 0, \quad  \lim_{n\to\infty}\sum_{|\alpha|=2k}\Vert \partial^\alpha m_{0}^* - \partial^\alpha m_{n}^*\Vert_{L^2(\Omega)} = 0.
\end{equation*}
Then, by the Gagliardo--Nirenberg interpolation inequality, we have $m_n^*  \to m_0^*$ in $H^{2k-2}(\Omega)$. 

Next, we show that $u_n^*$ converges to $u_0^*$ in $H^{2k}(\Omega)$. By \eqref{defBfmu}, \eqref{eq:bounds}, and  \eqref{eq:Lips}, we have 
\begin{equation}
\label{eq:eq2-6-1}
\begin{aligned}
&\ \epsilon\Big(\Vert u_0^*-u_n^*\Vert_{L^2(\Omega)}^2 + \sum_{|\alpha|=2k}\Vert
\partial^\alpha u_{0}^* - \partial^\alpha u_{n}^*\Vert_{L^2(\Omega)}^2\Big)\\
 = &\ B[u_0^*-u_n^*, u_0^*-u_n^*]
= 
\langle f_0-f_n, u_0^*-u_n^* \rangle \\
 \leq &\, \int_\Omega \big[|m_n - m_0||u_0^*-u_n^* | + |m_0 D_pH(x,Du_0) - m_n D_pH(x,Du_n)||Du_0^*-Du_n^*|\big]\,\dx\\
\leq &\, \int_\Omega \big[|m_n - m_0||u_0^*-u_n^* | + \tilde c|m_0  - m_n  ||Du_0^*-Du_n^*|+c|Du_0-Du_n||Du_0^*-Du_n^*|\big]\,\dx.
\end{aligned}
\end{equation}
Using  Gagliardo--Nirenberg interpolation inequality together with 
Young's inequality, we obtain from \eqref{eq:eq2-6-1} that
\begin{equation*}
\begin{aligned}
\Vert u_0^*-u_n^*\Vert_{L^2(\Omega)}^2 + \sum_{|\alpha|=2k}\Vert
\partial^\alpha u_{0}^* - \partial^\alpha u_{n}^*\Vert_{L^2(\Omega)} \leq C\big(\Vert m_0 - m_n \Vert_{L^2(\Omega)}^2 + \Vert D u_0 -D u_n \Vert_{L^2(\Omega)}^2\big)
\end{aligned}
\end{equation*}
for some constant \(C>0\)  independent of \(n\in\Nn\). Arguing as before, we conclude that  $u_n^* \to u_0^*$ in $H^{2k}(\Omega)$.

Finally, we address the compactness of $A$. We want to show that if \(\{(m_n,u_n)\}_{n=1}^\infty\) is a bounded sequence in 
$\widetilde\Aa \times H^{2k-1}(\Omega)$, then \(\{A(m_n,u_n)\}_{n=1}\)    is pre-compact  in 
$\widetilde\Aa \times H^{2k-1}(\Omega)$. This is an immediate consequence of
\eqref{eq:MorreyET}, Corollary~\ref{aprimvp1}, Lemma~\ref{ApriBL},
and the compact embedding  \(H^{2k}(\Omega) \times H^{2k}(\Omega) \hookrightarrow H^{2k-2}(\Omega)\times H^{2k-1}(\Omega)\)  due to the Rellich--Kondrachov theorem.   
\end{proof}

As we mentioned before, the existence of weak solutions to Problem~\ref{MP} follows from Schaefer's fixed-point Theorem. We state   next the version of this result that
we use here,  whose proof is a straightforward adaptation of the proof of Theorem~4, Section~9.2.2, in \cite{E6}. 
\begin{theorem}\label{thm:SFPT} Let \(X\) be a convex and closed subset of a Banach space with the property
that   \(\lambda w\in X\) whenever \(w\in X\) and \(\lambda \in [0,1] \). Assume that \(A:X \to X\) is a continuous and compact mapping such that the set 
\begin{equation*}
\begin{aligned}
\big\{w\in X  |\  w=\lambda A[w] \hbox{ for some } \lambda \in [0,1]\big\}
\end{aligned}
\end{equation*}
is bounded. Then, \(A\) has a fixed point.
\end{theorem}

\begin{pro}\label{ExiUniS1}
Consider Problem~\ref{MP}, let $A$ be the mapping defined in \eqref{OpeA}, and suppose that Assumptions \ref{Hcoer}--\ref{AWC1} hold for some \(\gamma>1\). Then, there exists a unique weak solution, $(m,u) \in H^{2k}(\Omega) \times H^{2k}(\Omega)$, to Problem~\ref{MP} with \(h=0\) in the sense of Definition~\ref{DMPWS1}.
\end{pro}
%
%
\begin{proof}
\textit{(Existence) }Fix $\lambda \in
[0,1]$, and let $(m_\lambda, u_\lambda)\in
\widetilde\Aa \times H^{2k-1}(\Omega)$
be such that
\begin{equation*}
   \begin{bmatrix}
      m_\lambda \\
      u_\lambda 
    \end{bmatrix}
  =  \lambda A
    \begin{bmatrix}
      m_\lambda  \\
      u_\lambda
    \end{bmatrix}.
\end{equation*}
If $\lambda=0$, then $(m_\lambda,u_\lambda)=(0,0)$.
Assume that $0<\lambda \leq1$; then,
by the definition of \(A\), Proposition~\ref{EMVP1},
Corollary~\ref{PVI}, and Proposition~\ref{ESBP1},
we have \(\frac{m_\lambda}\lambda \in\Aa\),
\(\frac{u_\lambda}\lambda \in H^{2k}_0(\Omega)\),
and
\begin{equation*}
\begin{aligned}
&\int_\Omega \lambda \big( -u_\lambda
- H(x,Du_\lambda) + g(m_\lambda) - V
\big)(w-m_\lambda) \, \dx\\
&\qquad+\int_{\Omega}  \epsilon m_\lambda(w-m_\lambda)+
  \epsilon \sum_{|\alpha|=2k}\partial^\alpha
m_\lambda(\partial^\alpha w-\partial^\alpha
m_\lambda)\, \dx
  \geq 0,\\ 
&\int_{\Omega} \Big[
   \lambda\big(m_\lambda -\div\big(m_\lambda
D_p H(x,Du_\lambda)\big)-\phi\big)v\,
\Big]\dx\\ 
&\qquad +\int_\Omega\bigg[
  \epsilon \Big(u_\lambda v + \sum_{|\alpha|=2k}
\partial^\alpha u_\lambda\partial^\alpha
v\Big)+\epsilon(\xi+\Delta^{2k}\xi)v
\bigg]\, \dx=0,
\end{aligned}
\end{equation*}
for all \(w\in\Aa\) and \(v\in H^{2k}_0(\Omega)\).
Hence, arguing as in Corollary~\ref{apriWS1},
we have
\begin{equation*}\label{E3.1}
\begin{aligned}
  &\int_{\Omega} \lambda \left[m_\lambda
g(m_\lambda)+m_\lambda |Du_\lambda|^\gamma+
 \phi |Du_{\lambda}|^\gamma \right]\,
\dx
 \\&\qquad+\epsilon \int_{\Omega} \Big[
m_\lambda^2+u_\lambda^2+\sum_{|\alpha|=2k}\big((\partial^\alpha
m_\lambda)^2+(\partial^\alpha u_{\lambda})^2\big)
 \Big] \, \dx
\leq C, 
\end{aligned}
\end{equation*}
where $C$ is a positive constant independent
of $\lambda$. Consequently, by
Assumption~\ref{AWC1} and the positivity
of \(m_\lambda\) and \(\phi\), we have

\begin{equation*}
\begin{aligned}
\epsilon \int_{\Omega} \Big[
m_\lambda^2+u_\lambda^2+\sum_{|\alpha|=2k}\big((\partial^\alpha
m_\lambda)^2+(\partial^\alpha u_{\lambda})^2\big)
 \Big] \, \dx
\leq C
\end{aligned}
\end{equation*}
where $C$ is another positive constant
independent
of $\lambda$.
Invoking the Gagliardo--Nirenberg interpolation
inequality,  we conclude that $(m_\lambda,
u_\lambda)$ is uniformly bounded in
$H^{2k}(\Omega) \times H^{2k}(\Omega)$
with respect to $\lambda$. This fact
and Proposition~\ref{ACC} allow us to
use Theorem~\ref{thm:SFPT} to conclude
that \(A \) has a fixed point, \((m,u)\in
\tilde \Aa \times H^{2k-1}(\Omega)\).
Finally, as before,  using the definition of
 \(A\), Proposition~\ref{EMVP1},
Corollary~\ref{PVI}, and Proposition~\ref{ESBP1},
we conclude that 
$(m,u)$ belongs to $  H^{2k}(\Omega)
\times H^{2k}(\Omega)$ and satisfies
 (E1)--(E3)
 with \(h=0\).
 
\textit{(Uniqueness) }
Assume that there are two fixed points, $(m_1,u_1)$ and $(m_2,u_2)$.
Taking $w=m_2$ in (E2) for \((u_1,m_1)\) and   $w=m_1$ in (E2) for \((u_2,m_2)\), and then summing the resulting inequalities, we have
\begin{equation}\label{ineqUni-1}
\begin{aligned}
&
\int_\Omega\big[u_1-u_2+H(x,Du_1)-H(x,Du_2)-(g(m_1)-g(m_2))\big](m_1-m_2)\,\dx\\
&\qquad-\int_\Omega \Big[\epsilon(m_1-m_2)^2+\epsilon\sum_{|\alpha|=2k}(\partial^{\alpha}m_1-\partial^{\alpha} m_2)^2\Big]\, \dx
\geq0.
\end{aligned}
\end{equation}
Because $u_1-u_2\in H_0^{2k}(\Omega)$, choosing $v=u_1-u_2$ in (E3) for  \((u_1,m_1)\) and \((u_2,m_2)\), and then subtracting the resulting equalities, we obtain
\begin{equation}\label{ineqUni-2}
\begin{aligned}
&\int_\Omega\Big(m_1-m_2-\div\big(m_1D_pH(x,Du_1)-m_2D_pH(x,Du_2)\big)\Big)(u_1-u_2)\,\dx\\
&\qquad+\int_\Omega\Big[\epsilon (u_1-u_2)^2+\epsilon\sum_{|\alpha|=2k}(\partial^\alpha u_1-\partial ^\alpha u_2)^2\Big]=0.
\end{aligned}
\end{equation}
Subtracting \eqref{ineqUni-1} from \eqref{ineqUni-2} and using Assumption \ref{Hmono}, we have
\begin{equation*}
\begin{aligned}
0\geq &
\int_\Omega\Big[
\epsilon(m_1-m_2)^2+\epsilon\sum_{|\alpha|=2k}(\partial^\alpha m_1-\partial^\alpha m_2)^2
+\epsilon(u_1-u_2)^2+\epsilon\sum_{|\alpha|=2k}(\partial^\alpha u_1-\partial^\alpha u_2)^2
\Big]\,\dx\\
&+
\left<
   F
  \begin{bmatrix}
      m_1  \\
      u_1 
  \end{bmatrix}
  -
   F
  \begin{bmatrix}
      m_2  \\
      u_2 
  \end{bmatrix},
  \begin{bmatrix}
      m_1  \\
      u_1
  \end{bmatrix}
  -
  \begin{bmatrix}
      m_2  \\
      u_2
  \end{bmatrix}
\right>
\geq 0.
\end{aligned}
\end{equation*}
Invoking Assumption~\ref{Hmono} once more, we conclude that  the integral in the preceding estimate is  equal to zero, from which we the identity $(m_1,u_1)=(m_2,u_2)$ follows.
\end{proof}

\begin{proof}[Proof of Theorem~\ref{MT}] 
Define, for \(x\in\Omega\) and \(p\in\Rr^d\),  $\widehat H(x,p):=H(x,p+Dh(x))$, $\widehat{V}(x):= V(x) + h(x)$, and 
$\widehat{\xi} (x):= \xi(x) +h(x)
$.

Note that \(H\) satisfies Assumptions~\ref{Hcoer}--\ref{Bderi} for some  \(\gamma>1\) if and only if  \(\widehat H\) satisfies Assumptions~\ref{Hcoer}--\ref{Bderi}
for the same \(\gamma\). Moreover, \((u,m)\in H^{2k}(\Omega)
\times H^{2k}(\Omega)\) satisfies (E1)--(E3)
if and only if \((\hat u,m):=(u -h,m)\in H^{2k}(\Omega)
\times H^{2k}(\Omega)\) satisfies (E1)--(E3) with \(h=0\) and with \(H\), \(V\), and \(\xi\) replaced by \(\widehat H\), \(\widehat V\), and \(\widehat \xi\), respectively. 

To conclude, it suffices to invoke Proposition~\ref{ExiUniS1}, which gives existence and uniqueness of a pair  \((\hat u,m)\in H^{2k}(\Omega)
\times H^{2k}(\Omega)\) satisfying (E1)--(E3) with \(h=0\) and
with \(H\), \(V\), and \(\xi\) replaced by \(\widehat H\), \(\widehat V\),
and \(\widehat \xi\), respectively. 
\end{proof}

\section{Proof of Theorem~\ref{TOP}}\label{PfMT2}
Here, we prove Theorem~\ref{TOP}. First, we study a compactness property of the unique weak solution to Problem~\ref{MP}. Then, we introduce a linear functional, $F_\epsilon$, corresponding to the equations~\eqref{HighdRP} in Problem~\ref{MP} and address its monotonicity. Finally, using Minty's method, we prove the existence of a weak solution to Problem~\ref{OP}.

\begin{lem}\label{WCWS1}
Consider Problem~\ref{MP} and suppose that Assumptions~\ref{Hcoer}--\ref{gwc}  hold for some \(\gamma>1\). Let $(m_\epsilon, u_\epsilon) \in H^{2k}(\Omega) \times H^{2k}(\Omega)$ be the unique weak solution to Problem~\ref{MP}. Then, there exists $(m,u)\in L^1(\Omega) \times W^{1,\gamma}(\Omega)$ such that $m \geq 0$, $u=h$ on $\partial\Omega$ in the sense of traces, and $(m_\epsilon, u_\epsilon)$ converges to $(m,u)$ weakly in $L^1(\Omega)\times W^{1,\gamma}(\Omega)$ as $\epsilon \to 0$, extracting a subsequence if necessary. 
\end{lem}
\begin{proof}
The existence of \(u\in W^{1,\gamma}(\Omega)\) as stated follows from the fact that  $u_\epsilon=h$ on $\partial\Omega$ in
the sense of traces and, by Corollary~\ref{apriDu1}, $u_\epsilon$ is uniformly bounded in $W^{1,\gamma}(\Omega) $ with respect to \(\epsilon\).

On the other hand, by Corollary~\ref{apriWS1}
and the positivity
of \(m_\epsilon\) and \(\phi\), we have
\begin{equation*}
\sup_{\epsilon\in(0,1)}
\int_\Omega m_\epsilon g(m_\epsilon)\, \dx <\infty.
\end{equation*}
Therefore, by Assumption~\ref{gwc}, there exists $m\in L^1(\Omega)$ such that $m_\epsilon \rightharpoonup m$ in $L^1(\Omega)$ as $\epsilon \to 0$, extracting a subsequence if necessary. Because $m_\epsilon\geq0$, we have $m\geq 0$.
\end{proof}

Fix $(\eta, v) \in H^{2k}(\Omega) \times H^{2k}(\Omega)$, let \(F[\eta,v]\) be the functional introduced in \eqref{DRF2},  and let 
 $F_{\epsilon} [\eta,v]: H^{2k}(\Omega) \times H^{2k}(\Omega) \to \Rr$ be the linear functional given by
\begin{equation}\label{DRF1}
\begin{aligned}
\left<
   F_\epsilon
  \begin{bmatrix}
      \eta \\
      v
  \end{bmatrix},
  \begin{bmatrix}
      w_1 \\
      w_2
  \end{bmatrix}
\right>
:= & \left<
   F
  \begin{bmatrix}
      \eta \\
      v
  \end{bmatrix},
  \begin{bmatrix}
      w_1 \\
      w_2
  \end{bmatrix}
\right> +\int_{\Omega}\Big(\epsilon \eta w_1 + \epsilon \sum_{|\alpha|=2k}\partial^\alpha \eta \partial^\alpha w_1 \Big)\,\dx
\\
&+\int_{\Omega}\Big[\epsilon (v + \xi )w_2 + \epsilon \sum_{|\alpha|=2k}\partial^\alpha (v+\xi)\partial^\alpha w_2\Big]\,\dx.
\end{aligned}
\end{equation}
Next, we prove the monotonicity of $F_\epsilon$ over $\Aa\times H^{2k}_h(\Omega)$, where, we recall, \(\Aa\) and 
$H_h^{2k}(\Omega)$ are given by \eqref{DAa}--\eqref{DHh2k}.
%
%
\begin{lem}\label{MORF}
Let $H,g,V,\phi,\xi$, and $h$ be as in Problem~\ref{MP}, let $F_\epsilon$ be given by \eqref{DRF1}, and suppose that Assumption~\ref{Hmono} holds. 
Then, for any $(\eta_1, v_1)$, $(\eta_2,  v_2) \in \Aa\times 
H_h^{2k}(\Omega)$, we have
\begin{equation*}
\left<
  F_{\epsilon}
  \begin{bmatrix}
      \eta_1  \\
      v_1 
  \end{bmatrix}
     -
      F_{\epsilon}
      \begin{bmatrix}
          \eta_2  \\
          v_2 
      \end{bmatrix},
      \begin{bmatrix}
          \eta_1  \\
          v_1
      \end{bmatrix}
    -
   \begin{bmatrix}
         \eta_2  \\
         v_2
  \end{bmatrix}
\right>
\geq 0.
\end{equation*}
\end{lem}
%
\begin{proof}
Let $(\eta_1, v_1)$, $(\eta_2,  v_2) \in \Aa\times 
H_h^{2k}(\Omega)$. Then, $v_1 - v_2 \in H_0^{2k}(\Omega)$; thus, using Assumption~\ref{Hmono} and integrating by parts, we obtain
\begin{align*}
&\left<
  F_{\epsilon}
  \begin{bmatrix}
      \eta_1  \\
      v_1 
  \end{bmatrix}
  -
  F_{\epsilon}
  \begin{bmatrix}
      \eta_2  \\
      v_2 
  \end{bmatrix},
  \begin{bmatrix}
      \eta_1  \\
      v_1
  \end{bmatrix}
  -
  \begin{bmatrix}
      \eta_2  \\
      v_2
  \end{bmatrix}
\right>\\
\geq&\int_{\Omega}
\epsilon\Big[(\eta_1-\eta_2)^2 + (v_1-v_2)^2+ \sum_{|\alpha|=2k}\big((\partial^\alpha \eta_1-\partial^\alpha \eta_2)^2 
 + (\partial^\alpha v_1-\partial^\alpha v_2)^2\big)\Big]\, \dx
\geq 0.\qedhere
\end{align*}
\end{proof}


\begin{proof}[Proof of Theorem~\ref{TOP}]
Let $(m_\epsilon, u_\epsilon)\in H^{2k}(\Omega) \times H^{2k}(\Omega)$ be the unique weak solution to Problem~\ref{MP} in the sense of Definition~\ref{DMPWS1}. Fix  $(\eta, v)\in \Aa \times H_h^{2k}(\Omega)$. By (E2) and (E3),  we have
\begin{equation*}
\left<
  F_\epsilon
  \begin{bmatrix}
      m_\epsilon \\
      u_\epsilon
  \end{bmatrix},
  \begin{bmatrix}
      \eta  \\
      v 
  \end{bmatrix}
  -
  \begin{bmatrix}
      m_\epsilon \\
      u_\epsilon
  \end{bmatrix}
\right>
\geq  0.
\end{equation*}
Thus, by Lemma~\ref{MORF},
\begin{equation}\label{eq:pfe}
\begin{aligned}
0&\leq \left<
  F_{\epsilon}
  \begin{bmatrix}
      \eta \\
      v
  \end{bmatrix}-F_{\epsilon} 
  \begin{bmatrix}
      m_\epsilon \\
      u_\epsilon
  \end{bmatrix},
  \begin{bmatrix}
      \eta  \\
      v 
  \end{bmatrix}
  -
  \begin{bmatrix}
      m_\epsilon \\
      u_\epsilon
  \end{bmatrix}
\right> \leq \left<
  F_{\epsilon}
  \begin{bmatrix}
      \eta \\
      v
  \end{bmatrix},
  \begin{bmatrix}
      \eta  \\
      v 
  \end{bmatrix}
  -
  \begin{bmatrix}
      m_\epsilon \\
      u_\epsilon
  \end{bmatrix}
\right>= \left<
  F
  \begin{bmatrix}
      \eta \\
      v
  \end{bmatrix},
  \begin{bmatrix}
      \eta  \\
      v 
  \end{bmatrix}
  -
  \begin{bmatrix}
      m_\epsilon \\
      u_\epsilon
  \end{bmatrix}
\right> + c_\epsi,
\end{aligned}
\end{equation}
where
\begin{equation*}
\begin{aligned}
c_\epsi:= &\int_{\Omega}\Big({\epsilon} \eta (\eta- m_\epsi) +  {\epsilon}\sum_{|\alpha|=2k}{}\partial^\alpha
\eta \partial^\alpha (\eta- m_\epsi) \Big)\,\dx\\ &
+\int_{\Omega}\Big[{\epsilon}(v + \xi )(v- u_\epsi)+ \epsilon \sum_{|\alpha|=2k}\partial^\alpha
(v+\xi)\partial^\alpha(v- u_\epsi)\Big]\,\dx.
\end{aligned}
\end{equation*}

By H\"older's inequality and Corollary~\ref{aprisqrt}, we conclude that
\begin{equation}
\begin{aligned}\label{eq:lim1}
\lim_{\epsi\to0} c_\epsi = 0.
\end{aligned}
\end{equation}

On the other hand, by  
Lemma~\ref{WCWS1}, there exists $(m,u)\in L^1(\Omega) \times
W^{1,\gamma}(\Omega)$  satisfying  (D1) and such that $(m_\epsilon, u_\epsilon)$ converges to $(m,u)$
weakly in $L^1(\Omega)\times W^{1,\gamma}(\Omega)$ as $\epsilon \to 0$, extracting a subsequence if necessary.
Then,
using the definition of \(F[\eta,v]\)
(see \eqref{DRF2}), we get
\begin{equation}
\label{eq:lim2}
\begin{aligned}
\lim_{\epsi\to0}  \left<
  F
  \begin{bmatrix}
      \eta \\
      v
  \end{bmatrix},
  \begin{bmatrix}
      \eta  \\
      v 
  \end{bmatrix}
  -
  \begin{bmatrix}
      m_\epsilon \\
      u_\epsilon
  \end{bmatrix}
\right>  = \left<
  F
  \begin{bmatrix}
      \eta \\
      v
  \end{bmatrix},
  \begin{bmatrix}
      \eta  \\
      v 
  \end{bmatrix}
  -
  \begin{bmatrix}
      m \\
      u
  \end{bmatrix}
\right>.
\end{aligned}
\end{equation}

From \eqref{eq:pfe}, \eqref{eq:lim1}, and \eqref{eq:lim2}, we conclude that
\begin{equation*}
\begin{aligned}
 \left<
  F
  \begin{bmatrix}
      \eta \\
      v
  \end{bmatrix},
  \begin{bmatrix}
      \eta  \\
      v 
  \end{bmatrix}
  -
  \begin{bmatrix}
      m \\
      u
  \end{bmatrix}
\right>\geq 0;
\end{aligned}
\end{equation*}
that is, \((m,u)\) also satisfies (D2). Hence, \((m,u)\) is a weak solution of Problem~\ref{OP} in the sense of Definition~\ref{DOPWS1}.
\end{proof}

\bibliographystyle{plain}
\bibliography{mfg.bib}
%


%
%


\end{document}